\documentclass[12pt]{amsart}
\usepackage{fourier}
\usepackage[T1]{fontenc}
\usepackage{graphics}
\usepackage{amsfonts,amssymb,color}
\usepackage[mathscr]{eucal}
\usepackage{amsmath, amsthm}
\usepackage{mathrsfs}
\usepackage{amsbsy}
\usepackage{bbm}
\usepackage{wasysym}
\usepackage{stmaryrd}
\usepackage{url}

\input xypic
\xyoption{all}

\makeindex
\makeglossary

\begin{document}
\baselineskip = 16pt

\newcommand \ZZ {{\mathbb Z}}
\newcommand \NN {{\mathbb N}}
\newcommand \RR {{\mathbb R}}
\newcommand \PR {{\mathbb P}}
\newcommand \AF {{\mathbb A}}
\newcommand \GG {{\mathbb G}}
\newcommand \QQ {{\mathbb Q}}
\newcommand \bcA {{\mathscr A}}
\newcommand \bcC {{\mathscr C}}
\newcommand \bcD {{\mathscr D}}
\newcommand \bcF {{\mathscr F}}
\newcommand \bcG {{\mathscr G}}
\newcommand \bcH {{\mathscr H}}
\newcommand \bcM {{\mathscr M}}
\newcommand \bcJ {{\mathscr J}}
\newcommand \bcL {{\mathscr L}}
\newcommand \bcO {{\mathscr O}}
\newcommand \bcP {{\mathscr P}}
\newcommand \bcQ {{\mathscr Q}}
\newcommand \bcR {{\mathscr R}}
\newcommand \bcS {{\mathscr S}}
\newcommand \bcU {{\mathscr U}}
\newcommand \bcV {{\mathscr V}}
\newcommand \bcW {{\mathscr W}}
\newcommand \bcX {{\mathscr X}}
\newcommand \bcY {{\mathscr Y}}
\newcommand \bcZ {{\mathscr Z}}
\newcommand \goa {{\mathfrak a}}
\newcommand \gob {{\mathfrak b}}
\newcommand \goc {{\mathfrak c}}
\newcommand \gom {{\mathfrak m}}
\newcommand \gon {{\mathfrak n}}
\newcommand \gop {{\mathfrak p}}
\newcommand \goq {{\mathfrak q}}
\newcommand \goQ {{\mathfrak Q}}
\newcommand \goP {{\mathfrak P}}
\newcommand \goM {{\mathfrak M}}
\newcommand \goN {{\mathfrak N}}
\newcommand \uno {{\mathbbm 1}}
\newcommand \Le {{\mathbbm L}}
\newcommand \Spec {{\rm {Spec}}}
\newcommand \Gr {{\rm {Gr}}}
\newcommand \Pic {{\rm {Pic}}}
\newcommand \Jac {{{J}}}
\newcommand \Alb {{\rm {Alb}}}
\newcommand \Corr {{Corr}}
\newcommand \Chow {{\mathscr C}}
\newcommand \Sym {{\rm {Sym}}}
\newcommand \Prym {{\rm {Prym}}}
\newcommand \cha {{\rm {char}}}
\newcommand \eff {{\rm {eff}}}
\newcommand \tr {{\rm {tr}}}
\newcommand \Tr {{\rm {Tr}}}
\newcommand \pr {{\rm {pr}}}
\newcommand \ev {{\it {ev}}}
\newcommand \cl {{\rm {cl}}}
\newcommand \interior {{\rm {Int}}}
\newcommand \sep {{\rm {sep}}}
\newcommand \td {{\rm {tdeg}}}
\newcommand \alg {{\rm {alg}}}
\newcommand \im {{\rm im}}
\newcommand \gr {{\rm {gr}}}
\newcommand \op {{\rm op}}
\newcommand \Hom {{\rm Hom}}
\newcommand \Hilb {{\rm Hilb}}
\newcommand \Sch {{\mathscr S\! }{\it ch}}
\newcommand \cHilb {{\mathscr H\! }{\it ilb}}
\newcommand \cHom {{\mathscr H\! }{\it om}}
\newcommand \colim {{{\rm colim}\, }} 
\newcommand \End {{\rm {End}}}
\newcommand \coker {{\rm {coker}}}
\newcommand \id {{\rm {id}}}
\newcommand \van {{\rm {van}}}
\newcommand \spc {{\rm {sp}}}
\newcommand \Ob {{\rm Ob}}
\newcommand \Aut {{\rm Aut}}
\newcommand \cor {{\rm {cor}}}
\newcommand \Cor {{\it {Corr}}}
\newcommand \res {{\rm {res}}}
\newcommand \red {{\rm{red}}}
\newcommand \Gal {{\rm {Gal}}}
\newcommand \PGL {{\rm {PGL}}}
\newcommand \Bl {{\rm {Bl}}}
\newcommand \Sing {{\rm {Sing}}}
\newcommand \spn {{\rm {span}}}
\newcommand \Nm {{\rm {Nm}}}
\newcommand \inv {{\rm {inv}}}
\newcommand \codim {{\rm {codim}}}
\newcommand \Div{{\rm{Div}}}
\newcommand \sg {{\Sigma }}
\newcommand \DM {{\sf DM}}
\newcommand \Gm {{{\mathbb G}_{\rm m}}}
\newcommand \tame {\rm {tame }}
\newcommand \znak {{\natural }}
\newcommand \lra {\longrightarrow}
\newcommand \hra {\hookrightarrow}
\newcommand \rra {\rightrightarrows}
\newcommand \ord {{\rm {ord}}}
\newcommand \Rat {{\mathscr Rat}}
\newcommand \rd {{\rm {red}}}
\newcommand \bSpec {{\bf {Spec}}}
\newcommand \Proj {{\rm {Proj}}}
\newcommand \pdiv {{\rm {div}}}
\newcommand \CH {{\it {CH}}}
\newcommand \wt {\widetilde }
\newcommand \ac {\acute }
\newcommand \ch {\check }
\newcommand \ol {\overline }
\newcommand \Th {\Theta}
\newcommand \cAb {{\mathscr A\! }{\it b}}

\newenvironment{pf}{\par\noindent{\em Proof}.}{\hfill\framebox(6,6)
\par\medskip}

\newtheorem{theorem}[subsection]{Theorem}
\newtheorem{conjecture}[subsection]{Conjecture}
\newtheorem{proposition}[subsection]{Proposition}
\newtheorem{lemma}[subsection]{Lemma}
\newtheorem{remark}[subsection]{Remark}
\newtheorem{remarks}[subsection]{Remarks}
\newtheorem{definition}[subsection]{Definition}
\newtheorem{corollary}[subsection]{Corollary}
\newtheorem{example}[subsection]{Example}
\newtheorem{examples}[subsection]{examples}

\title{Theta divisors of abelian varieties and push-forward homomorphism at the level of Chow groups}
\author{ Kalyan Banerjee}

\address{Indian Statistical Institute, Bangalore Center, Bangalore 560059}


\email{kalyanb$_{-}$vs@isibang.ac.in}
\footnotetext{Mathematics Classification Number: 14C25, 14D05, 14D20,
 14D21}
\begin{abstract}
In this text we prove that if an abelian variety $A$ admits of an embedding into the Jacobian of a smooth projective curve $C$, and if we consider $\Th_A$ to be the divisor $\Th_C\cap A$, where $\Th_C$ denotes the theta divisor of $J(C)$, then the embedding of $\Th_A$ into $A$ induces an injective push-forward homomorphism at the level of Chow groups. We show that this is the case for every principally polarized abelian varieties.
\end{abstract}

\maketitle

\section{Introduction}
In the paper \cite{BI} the authors were investigating the following question. Let $C$ be a smooth projective curve of genus $g$ and let $\Th$ denote the theta divisor embedded into the Jacobian $J(C)$ of the curve $C$. Let $j$ denote this embedding. Then the push-forward homomorphism $j_*$ at the level of Chow groups is injective. Also in this paper the author discussed about the push-forward homomorphism at the level of Chow groups induced by the closed embedding of some special divisors in the Jacobian $J(C)$, arising from finite, \'etale coverings of the curve $C$, see \cite{BI}[theorem $4.1$].

In this paper we investigate the following question for an arbitrary principally polarized abelian variety $A$ . That is let $A$ be a principally polarized abelian variety and let $H$ denote a divisor embedded inside $A$. Let $j$ denote this embedding. Then can we say that the push-forward homomorphism $j_*$ at the level of Chow groups of $k$-dimensional cycles ($k\geq 0$) is injective? This question is affirmatively answered in the case when the abelian variety $A$ is embedded in the Jacobian variety  and we consider the divisor $\Th\cap A$ inside $A$, where $\Th$ is the theta divisor of $J(C)$. This is exactly the case of Prym-Tyurin varieties, which are abelian varieties embedded inside some Jacobian variety, and the intersection of the theta divisor of $J(C)$ with $A$ is linearly equivalent to some multiple of the Theta divisor of $A$. Since any principally polarized abelian variety is a Prym-Tyurin variety of some exponent (see \cite{BL} corollary $12.2.4$), so for the principally polarized abelian varieties the above question about the injectivity of the push-forward homomorphism at the level of Chow groups of $k$-dimensional cycle ($k\geq0$) is  answered, when the divisor $H$ is the intersection of the abelian variety $A$ with the theta divisor of the ambient Jacobian variety, where $A$ is embedded.

\textit{Let $A$ be a principally polarized abelian variety embedded into $J(C)$ for some smooth projective curve $C$. Let $\Theta$ be the theta divisor of $J(C)$. Then  the embedding of $\Theta\cap A$ into $A$ induces an injective push-forward homomorphism at the level of Chow groups of $k$-cycle with $k\geq0$.}

As an application we get that the embedding of the theta divisor inside a Prym variety induces injection at the level of Chow groups. We also show that if we start with an principally polarized abelian surface $A$ and consider the corresponding $K3$-surface, then the push-forward at the level of Chow groups induced by the closed embedding of the divisor coming from $\Th\cap A$ inside the $K3$-surface is injective.

{\small \textbf{Acknowledgements:} The author is indebted to A.Beauville for communicating the main application of the paper. The author expresses his sincere gratitude to Jaya Iyer for telling this problem about injectivity of push-forward induced by the closed embedding of a divisor into an abelian variety, to the author and for discussions relevant to the theme of the paper. The author thanks C.Voisin for pointing out the fact that this injectivity of the push-forward is not true for  zero cycles on a very general divisor on an abelian variety and also for relevant discussion regarding the theme of the paper. Also the author wishes to thank the ISF-UGC grant for funding this project and hospitality of Indian Statistical Institute, Bangalore Center  for hosting this project.}

\section{Abelian varieties embedded in Jacobians}
Let $A$ be an abelian variety embedded inside the Jacobian of a smooth projective curve $C$. Let $\Th_C$ denote the theta divisor of the Jacobian $J(C)$. Consider $\Th_A$ to be $\Th_C\cap A$, and the closed embedding of $\Th_A$ into $A$, denote it by $j_A$, then we prove that $j_{A*}$ is injective from $\CH_*(\Th_A)$ to $\CH_*(A)$.

To prove that first we show that the embedding of $\Th_C$ into $J(C)$ gives rise to an injection at the level of Chow groups. Although this has been proved in \cite{BI}[theorem 3.1], here we present an alternative proof following \cite{Collino} which gives a better understanding of the picture when we blow up $J(C)$ along some subvariety (in our case finitely many points).

\begin{theorem}
The embedding of the theta divisor $\Th_C$ into $J(C)$ for a smooth projective curve $C$, gives rise to an injection at the level of Chow groups.
\end{theorem}
\begin{proof}
We use the fact that $\Sym^g C$ maps surjectively and birationally onto $J(C)$ and $\Sym^{g-1}C$ maps surjectively and birationally onto $\Th_C$. We have a natural correspondence $\Gamma$ given by $\pi_g\times \pi_{g-1}(Graph(pr))$, where $pr$ is the projection from $C^g$ to $C^{g-1}$, $\pi_i$ is the natural morphism from $C^i$ to $\Sym^i C$. Consider the correspondence $\Gamma_1$ on $J(C)\times \Th_C$ given by
$(f_1\times f_2)(\Gamma)$,
where $f_1,f_2$ are natural morphisms from $\Sym^{g-1}C,\Sym^g C$ to $\Th_C,J(C)$ respectively. Then by projection formula it follows that
$${\Gamma_1}_*j_*$$
is induced by $(j\times id)^*(\Gamma_1)$
where $j$ is the closed embedding of $\Th_C$ into $J(C)$. Now we compute the cycle
$$(j\times id)^*(\Gamma_1)$$
that is nothing but the collection of divisors
$$(D_1,D_2)$$
such that $D_1$ is linearly equivalent to $\sum_{i=1}^{g-1} x_i-(g-1)p$ and $D_2$ is linearly equivalent to $\sum_{i=1}^{g-1}y_i-(g-1)p$, where
$$([x_1,\cdots,x_{g-1},p],[y_1,\cdots,y_{g-1}])\in \Gamma\;.$$
Therefore without loss of generality we can assume that elements in $(j\times id)^*(\Gamma_1)$ are classes of effective divisors on $C$ of the form
$$([x_1+\cdots+x_{g-1}+p-gp],[y_1+\cdots+y_{g-1}+p-gp])$$
such that
$$([x_1,\cdots,x_{g-1},p],[y_1,\cdots,y_{g-1}])\in \Gamma$$
so we get that either
$$x_i=y_i$$
for all $i$ or
$$y_i=p$$
for some $i$. Therefore we get that $(j\times id)^*\Gamma_1$ is equal to
$$\Delta+Y$$
where $Y$ is supported on $\Th_C\times \sum_{i=1}^{g-2}C_i$. The fact that the multiplicity of $\Delta$ in $(j\times id)^*(\Gamma_1)$ is $1$ follows from the fact that $\Sym^g C$ maps surjectively and birationally onto $J(C)$, and the computations following \cite{Collino}. Also here the Chow moving lemma holds for $\Th_C\times \Th_C$, because it holds for $\Sym^{g-1}C\times \Sym^{g-1}C$ with the cycles taken with $\QQ$-coefficients and the fact that $f_1$ is birational. So let
$$\rho:U\to \Th_C$$
be the open embedding of the complement of $\sum_{i=1}^{g-2}C_i$ in $\Th_C$. Then we have
$$\rho^*{\Gamma_1}_*j_*(Z)=\rho^*(Z+Z_1)=\rho^*(Z)$$
where $Z_1$ is supported on $\sum_{i=1}^{g-2}C_i$. This follows since $(j\times id)^*(\Gamma_1)=\Delta+Y$, where $Y$ is supported on $\Th_C\times\sum_{i=1}^{g-2}C_i$. Now consider the following commutative diagram.
$$
  \xymatrix{
     \CH_*(\sum_{i=1}^{g-2}C_i) \ar[r]^-{j'_{*}} \ar[dd]_-{}
  &   \CH_*(\Th_C) \ar[r]^-{\rho^{*}} \ar[dd]_-{j_{*}}
  & \CH_*(U)  \ar[dd]_-{}  \
  \\ \\
   \CH_*(\sum_{i=1}^{g-2}C_i) \ar[r]^-{j''_*}
    & \CH_*(J(C)) \ar[r]^-{}
  & \CH^*(V)
  }
$$
Here $U,V$ are complements of $\sum_{i=1}^{g-2}C_i$ in $\Th_C,J(C)$ respectively. Now suppose that $j_*(z)=0$, then from the previous it follows that
$$\rho^*{\Gamma_1}_*j_*(z)=\rho^*(z)=0$$
by the localisation exact sequence it follows that there exists
$z'$ in $\sum_{i=1}^{g-2}C_i$ such that $j'_*(z')=z$. By the commutativity and the induction hypothesis it follows that
$$j''_*(z')=0$$
since $\sum_{i=1}^{g-2}C_i$ is of dimension $g-2$. So we get that $z'=0$ hence $z=0$. So $j_*$ is injective.
\end{proof}

Now we prove the following:

\begin{theorem}
\label{theorem2} Let $A$ be a principally polarized abelian variety embedded into Jacobian $J(C)$ of a smooth projective curve $C$. Let $\Th_C$ denote the theta divisor of $J(C)$. Pull it back to $A$ and denote the pull-back by $\Th_A$. Suppose that $\Sym^{g-i}C\cap A'$ is smooth for all $i\geq 0$, where $A'$ is the inverse image of $A$, under the map $\Sym^g C\to J(C)$. Then the closed embedding $j:\Th_A\to A$ induces an injection $j_*$ at the level of Chow groups.
\end{theorem}
\begin{proof} We have a natural correspondence as previous, $\Gamma$ given by $\pi_g\times \pi_{g-1}(Graph(pr))$ on $\Sym^g C\times \Sym^{g-1}C$, where $pr$ is the projection from $C^g$ to $C^{g-1}$, $\pi_i$ is the natural morphism from $C^i$ to $\Sym^i C$. Consider the correspondence $\Gamma_1$ on $J(C)\times \Th_C$ given by
$(f_1\times f_2)(\Gamma)$,
where $f_1,f_2$ are natural morphisms from $\Sym^{g-1}C,\Sym^g C$ to $\Th_C,J(C)$ respectively. Consider the restriction of $\Gamma_1$ to $A\times \Th_A$. Call it $\Gamma_1'$. Let $j$ denote the embedding of $\Th_A$ into $A$. Then by projection formula it follows that
$${\Gamma_1'}_*j_*$$
is induced by $(j\times id)^*(\Gamma_1')$,
this is because $A$ is smooth and hence $j$ is a local complete intersection. Now we compute the cycle
$$(j\times id)^*(\Gamma_1')$$
that is nothing but the collection of divisors
$$(D_1,D_2)$$
such that $D_1$ is linearly equivalent to $\sum_{i=1}^{g-1} x_i-(g-1)p$ and $D_2$ is linearly equivalent to $\sum_{i=1}^{g-1}y_i-(g-1)p$, where
$$([x_1,\cdots,x_{g-1},p],[y_1,\cdots,y_{g-1}])\in \Gamma'\;.$$
Here $\Gamma'$ is the restriction of $\Gamma$ to the scheme-theoretic inverse $A'\times \Th_A'$ of $A\times \Th_A$, under the natural map from $\Sym^g C\times \Sym^{g-1}C$ to $J(C)\times \Th_C$.
Therefore without loss of generality we can assume that elements in $(j\times id)^*(\Gamma_1')$ are classes of effective divisors on $C$ of the form
$$([x_1+\cdots+x_{g-1}+p-gp],[y_1+\cdots+y_{g-1}+p-gp])$$
such that
$$([x_1,\cdots,x_{g-1},p],[y_1,\cdots,y_{g-1}])\in \Gamma'$$
so we get that either
$$x_i=y_i$$
for all $i$ or
$$y_i=p$$
for some $i$. Therefore we get that $(j\times id)^*\Gamma_1'$ is equal to
$$\Delta_{\Th_A\times \Th_A}+Y$$
where $Y$ is supported on $\Th_A\times \sum_{i=1}^{g-2}C_i\cap A$. The fact that the multiplicity of $\Delta$ in $(j\times id)^*(\Gamma_1')$ is $1$ follows from the fact that $\Sym^g C$ maps surjectively and birationally onto $J(C)$, and the computations following \cite{Collino}. Also here the Chow moving lemma holds for $\Th_A\times \Th_A$ by the assumption of the theorem.  So let
$$\rho:U\to \Th_A$$
be the open embedding of the complement of $\sum_{i=1}^{g-2}C_i\cap A$ in $\Th_A$. Then we have
$$\rho^*{\Gamma_1'}_*j_*(Z)=\rho^*(Z+Z_1)=\rho^*(Z)$$
where $Z_1$ is supported on $\sum_{i=1}^{g-2}C_i\cap A$. This follows since $(j\times id)^*(\Gamma_1')=\Delta+Y$, where $Y$ is supported on $\Th_A\times\sum_{i=1}^{g-2}C_i$. Now consider the following commutative diagram.
$$
  \xymatrix{
     \CH_*(\sum_{i=1}^{g-2}C_i\cap A) \ar[r]^-{j'_{*}} \ar[dd]_-{}
  &   \CH_*(\Th_A) \ar[r]^-{\rho^{*}} \ar[dd]_-{j_{*}}
  & \CH_*(U)  \ar[dd]_-{}  \
  \\ \\
   \CH_*(\sum_{i=1}^{g-2}C_i\cap A) \ar[r]^-{j''_*}
    & \CH_*(A) \ar[r]^-{}
  & \CH^*(V)
  }
$$
Here $U,V$ are complements of $\sum_{i=1}^{g-2}C_i\cap A$ in $\Th_A,A$ respectively. Now suppose that $j_*(z)=0$, then from the previous it follows that
$$\rho^*{\Gamma_1'}_*j_*(z)=\rho^*(z)=0$$
by the localisation exact sequence it follows that there exists
$z'$ supported on $\sum_{i=1}^{g-2}C_i\cap A$ such that $j'_*(z')=z$. By the commutativity and the induction hypothesis it follows that
$$j''_*(z')=0$$
since $\sum_{i=1}^{g-2}C_i\cap A$ is of dimension $d-2$, here $d$ is the dimension of $A$. So we get that $z'=0$ hence $z=0$. So $j_*$ is injective.

\end{proof}

The previous theorem gives rise to the following corollary:

\begin{corollary}
Let $\wt{C}\to C$ be an unramified double cover of smooth projective curves. Consider the Prym variety associated to this double cover, denote by $P(\wt{C}/C)$. Consider the embedding of $P(\wt{C}/C)$ into $J(\wt{C})$. Let $\Th'$ be the pullback of the theta divisor on $J(\wt{C})$ to $P(\wt{C}/C)$. Then the closed embedding $\Th'\to P(\wt{C}/C)$ induces an injection at the level of Chow groups.
\end{corollary}

\begin{proof}
Let $g$ be the genus of $C$. So by Riemann-Hurwitz formula the genus of $\wt{C}$ is $2g-1$. We have the following commutative square.

$$
  \diagram
 \Sym^{2g-1} \wt{C}\ar[dd]_-{\theta_C} \ar[rr]^-{} & & \Sym^{2g-1} C \ar[dd]^-{\theta_{\wt{C}}} \\ \\
  J(\wt{C}) \ar[rr]^-{} & & J(C)
  \enddiagram
  $$
 Then first of all the Prym variety is the image under $\theta_C$ of the double cover $P'$ of a projective space $\PR^g$. This follows from the Riemann-Roch and the very definition of the Prym variety. Now consider the intersection of $\Sym^{2g-i}\wt{C}$ with $P'$, where $i\geq 2$. This intersection is smooth for a general copy of $\Sym^{2g-i}\wt{C}$ in $\Sym^{2g-1}\wt{C}$, in the following way. Consider the family
 $$\bcU:=\{([x_1,\cdots,x_{2g-2},x_{2g-1}],p)\in P'\times \wt{C}|p\in [x_1,\cdots,x_{2g-1}]\}$$
 and the projection from $\bcU$ to $\wt{C}$. Then $\bcU$ is a family of $\Sym^{2g-2}\wt{C}\cap P'$ over $\wt{C}$. Hence by Bertini's theorem for a general $p$, $\bcU_p$ is smooth. Similarly we can prove that a general $\Sym^{2g-i}\wt{C}\cap P$ is smooth. Hence we have the assumption of the Theorem  \ref{theorem2} is satisfied, whence the conclusion follows.
\end{proof}

Now we prove that if we blow up $J(C)$ at finitely many points and denote the blow up by $\wt{J(C)}$ and let $\wt{\Th_C}$ denote the total transform of $\Th_C$, then the closed embedding of $\wt{\Th_C}$ into $\wt{J(C)}$ induces injective push-forward homomorphism at the level of Chow groups.

\begin{theorem}
\label{theorem3}
Let $\wt{J(C)}$ be the blow up of $J(C)$ at some non-singular subvariety $Z$ whose inverse image is $E$. Let $\Th_C$ intersect $Z$ transversely. Let $\wt{\Th_C}$ denote the strict transform of $\Th_C$ in $\wt{J(C)}$. Then the closed embedding of $\wt{\Th_C}$ into $\wt{J(C)}$ induces injective push-forward homomorphism at the level of Chow groups of zero cycles.
\end{theorem}

\begin{proof}
Let us consider $\pi$ to be the morphism from $\wt{J(C)}$ to $J(C)$. Consider the correspondence $(\pi'\times \pi)^*(\Gamma_1)$, where $\pi'$ is the restriction of $\pi$ to $\wt{\Th_C}$. Call this correspondence $\Gamma'$. Then $\Gamma'_*\wt{j}_*$ is induced by $(\wt{j}\times id)^*\Gamma'$, where $\wt{j}$ is the closed embedding of $\wt{\Th_C}$ into $\wt{J(C)}$. Consider the commutative square.
$$
  \diagram
   \wt{\Th_C}\times \wt{\Th_C}\ar[dd]_-{\pi'\times \pi'} \ar[rr]^-{\wt{j}\times id} & & \wt{J(C)}\times \wt{J(C)} \ar[dd]^-{\pi\times \pi} \\ \\
  \Th_C\times \Th_C \ar[rr]^-{j\times id} & & J(C)\times J(C)
  \enddiagram
  $$

This gives us that
$$(\wt{j}\times \id)^*\Gamma'=(\pi'\times \pi')^*(j\times id)^*\Gamma_1=(\pi'\times \pi')^*(\Delta+Y)$$
where $Y$ is supported on $\Th_C\times \sum_{i=1}^{g-2}C_i$. Now
$$(\pi'\times \pi')^*(\Delta)=\Delta+V$$
where $E$ is the exceptional locus of $\pi$ and $V$ is supported on $(E\cap \wt{\Th_C})\times (E\cap \wt{\Th_C})$. So considering $\rho$ to be the inclusion of the complement of $\wt{\sum_{i=1}^{g-2}C_i}$ in $\wt{\Th_C}$  and applying Chow moving lemma we have
$$\rho^*\Gamma'_*\wt{j}_*=\rho^*\;.$$
Consider the following commutative diagram.
$$
  \xymatrix{
     \CH_0(A) \ar[r]^-{\wt{j'}_{*}} \ar[dd]_-{}
  &   \CH_0(\wt{\Th_C}) \ar[r]^-{\rho_0^{*}} \ar[dd]_-{\wt{j}_{*}}
  & \CH_0(U)  \ar[dd]_-{}  \
  \\ \\
   \CH_0(A) \ar[r]^-{\wt{j''}_*}
    & \CH_0(\wt J(C)) \ar[r]^-{}
  & \CH_0(V)
  }
$$
Here $A= \wt{\sum_{i=1}^{g-2}C_i}$.
Now suppose that $\wt{j}_*(z)=0$. By the previous computation we get that $\rho^*(z)=0$, so by the localisation exact sequence we get that there exists $z'$ in $\CH_*(A)$ such that $\wt{j'}_*(z')=z$.

By induction $\CH_*(A)\to \CH_*(\wt{J(C)})$ is injective. So we get that $z'=0$ hence $z=0$ giving $\wt{j_*}$ injective.

\end{proof}

Now let $A$ be an abelian surface which is embedded in some $J(C)$. Let $i$ denote the involution of $A$. Then $i$ has $16$ fixed points. We blow up $A$ along these fixed points. Then we get $\wt{A}$ on which we have an induced involution, call it $i$. Let $\wt{\Th_A}$ denote the total transform of $\Th_A$ in $\wt{A}$. Then the above discussion tells us the following.

\begin{theorem}
The closed embedding of $\wt{\Th_A}$ into $\wt{A}$ induces injective push-forward homomorphism at the level of Chow groups of zero cycles.
\end{theorem}

Now $\wt{A}/i$ is the Kummer's K3 surface associated to $A$. Suppose that we choose $\Th_C$ such that it is $i$ invariant. Then $\wt{\Th_A}$ will be $i$ invariant. The above theorem gives us:

\begin{theorem}
The closed embedding of $\wt{\Th_A}/i$ into $\wt{A}/i$ induces injective push-forward homomorphism at the level of Chow groups of zero cycles with $\QQ$-coefficients.
\end{theorem}

Note that all these techniques can be repeated if we consider the group of algebraic cycles modulo algebraic equivalence. Then therefore the closed embedding $\wt{\Th_A}/i$ into $\wt{A}/i$ induces injection at the level of zero cycles modulo algebraic equivalence.


\begin{thebibliography}{AAAAA}

\bibitem[BI]{BI}K.Banerjee, Jaya NN Iyer, {\em On the kernel of the push-forward homomorphism between Chow groups},{\small \tt arxiv::1504.07887},2015.
\bibitem[BAN]{BAN}K.Banerjee, {\em On the closed embedding of the product of theta divisors into  product of Jacobians  and Chow groups}, {\small \tt https://arxiv.org/abs/1609.04198}, 2015
\bibitem [B]{B}K.Banerjee, {\em Blow ups and base changes of symmetric powers and Chow groups}, {\small \tt https://arxiv.org/abs/1609.03844}, 2015.





\bibitem[BL]{BL} C.Birkenhake and H.Lange, {\em Complex abelian varieties}, Springer-Verlag Berlin, Heidelberg, New York, 2003.

\bibitem[Bl]{Bloch} S. Bloch, {\em Algebraic cycles and higher K-theory}, Advances in Mathematics, \textbf{61}, 267-304, 1986.

\bibitem[Co]{Collino} A. Collino, {\em The rational equivalence ring of symmetric product of curves}, Illinois journal of mathematics,19,no.4, 567-583, 1975.

\bibitem[Fu]{Fulton} W. Fulton, {\em Intersection theory}, Ergebnisse der Mathematik und ihrer Grenzgebiete (3),  2.
      Springer-Verlag, Berlin, 1984.
\bibitem[No]{Nori} M.V. Nori, {\em Algebraic cycles and Hodge-theoretic connectivity}, Invent. Math. 111 (1993), no. \textbf{2}, 349--373.

\bibitem[Pa]{Paranjape}  Paranjape, {\em Cohomological and cycle-theoretic connectivity}, Ann. of Math. (2) 139 (1994), no. \textbf{3}, 641--660.








\end{thebibliography}
\end{document}